\documentclass[12pt,thmsa]{amsproc}

\usepackage{amssymb}
\usepackage[dvips]{graphics}
\input{amssym.def}
\input{amssym.tex}
\usepackage{amscd,euscript,mathrsfs}

\def\cgnk{\Pi_{n,k}}

\def\Cal{\mathcal}

\def\H{{\Cal H}}

\def\F{{\Cal F}}

\def\bbr{{\Bbb R}}

\def\bbn{{\Bbb N}}
\def\bbh{{\Bbb H}}

\def\bbc{{\Bbb C}}
\def\bbz{{\Bbb Z}}

\def\bbe{{\Bbb E}}

\def\const{{\hbox{\rm const}}}

\def\cosh{{\hbox{\rm cosh}}}
\def\sinh{{\hbox{\rm sinh}}}

\def\Pr{{\hbox{\rm Pr}}}

\def\rn{\bbr^n}

\def\part{\partial}
\def\intl{\int\limits}
\def\b{\beta}

\def\Gam{\Gamma}

\def\a{\alpha}
\def\om{\omega}

\def\del{\delta}
\def\vp{\varphi}

\def\gam{\gamma}

\def\sig{\sigma}

\def\z{\zeta}
\def\e{\varepsilon}
\def\t{\tau}
\def\th{\theta}

\newtheorem{theorem}{Theorem}[section]
\newtheorem{lemma}[theorem]{Lemma}

\theoremstyle{definition}
\newtheorem{definition}[theorem]{Definition}

\theoremstyle{remark}
\newtheorem{remark}[theorem]{Remark}

\theoremstyle{corollary}
\newtheorem{corollary}[theorem]{Corollary}
\newtheorem{proposition}[theorem]{Proposition}

\numberwithin{equation}{section}

\newcommand{\be}{\begin{equation}}
\newcommand{\ee}{\end{equation}}

\newcommand{\bea}{\begin{eqnarray}}
\newcommand{\eea}{\end{eqnarray}}
\newcommand{\Bea}{\begin{eqnarray*}}
\newcommand{\Eea}{\end{eqnarray*}}

\def\sideremark#1{\ifvmode\leavevmode\fi\vadjust{\vbox to0pt{\vss
 \hbox to 0pt{\hskip\hsize\hskip1em
\vbox{\hsize2cm\tiny\raggedright\pretolerance10000
 \noindent #1\hfill}\hss}\vbox to8pt{\vfil}\vss}}}%

\begin{document}

\title[ Overdetermined Transforms in Integral Geometry]
{ Overdetermined Transforms in Integral Geometry}

\author{B. Rubin}
\address{
Department of Mathematics, Louisiana State University, Baton Rouge,
LA, 70803 USA}

\email{borisr@math.lsu.edu}




\dedicatory{Dedicated to Professor David Shoikhet
on the occasion of his 60th birthday}


\begin{abstract}
 A simple example  of an $n$-dimensional admissible complex of planes is given for the  overdetermined $k$-plane  transform in $\bbr^n$.  For the corresponding restricted $k$-plane  transform
   sharp existence conditions are obtained   and explicit inversion formulas are discussed  in the general context of $L^p$ functions. Similar questions are studied for overdetermined Radon type transforms on the sphere and the hyperbolic space. A theorem describing  the range of the restricted $k$-plane  transform on the space of rapidly decreasing smooth functions is proved.

\end{abstract}

\maketitle

\section{Introduction}

The $k$-plane Radon-John transform of a function $f$ on $\rn$  is a mapping
\be\label{i9435rf7}  R_k :\; f(x)  \rightarrow \vp(\t)=\intl_{\t} f(x)\, d_\t x,\ee
where  $\t$ is a $k$-dimensional plane in $\rn$, $1\le k\le n-1$, and  $ d_{\t} x$ is  the Euclidean volume element on $\t$; see, e.g., \cite{GGG03, GGG80, He11, Jo38}.
We denote by  $\cgnk$   the
 manifold of all non-oriented $k$-dimensional planes  in
$\rn$. Since  $\dim \Pi_{n, k}=(k+1)(n-k)$ is greater than $n$ if $k<n-1$, then the  inversion problem for $R_k$ is overdetermined if we use information about $(R_kf)(\t)$ for {\it all} $\t\in \Pi_{n, k}$. The celebrated Gel'fand's  question  is   how to reduce this overdeterminedness  or, more precisely, how to define an $n$-dimensional subset  $\tilde \Pi_{n,k}$  of $ \Pi_{n,k}$ so that  $f(x)$  could be recovered, knowing   $\vp(\t)$ only for $\t\in \tilde \Pi_{n,k}$; see, e.g., \cite {Gelf60}. We call the subsets $\tilde \Pi_{n,k}$  admissible complexes of $k$-planes.

The background of the theory related to this question was developed  in   Gel'fand's school; see, e.g.,   \cite {GGi77, GGG80, GGr61, GGr68, GGr91,   GGR84,  GGrS66, GGrS67,   Gon89, Ki61, Mai}. The construction of the admissible complexes in these works is usually given in topological terms,  and the relevant inversion formulas rely on the fundamental concepts of the kappa-operator (for $k$ even) and the Crofton operator (for $k$ odd, when the inversion formulas are nonlocal). An alternate approach to nonlocal inversion formulas, which employs  the Fourier integral operators, was suggested by  Greenleaf and  Uhlmann \cite{GU}.  In all aforementioned works the Radon type transforms are studied mainly on  smooth rapidly decreasing functions.

 Our interest to the  problem  is motivated by the following.

1. Is it possible to construct an easily visualizable admissible complex $\tilde \Pi_{n,k}$ using relatively simple tools and derive the relevant explicit inversion formulas for $R_kf$? A progress in this direction  might be helpful in  convex geometry, where geometrically transparent  analytic constructions are crucial; see, e.g., \cite{G06, Ru09}.

2. It is known \cite{So79, Str81, Ru04b} that for $f\in L^p (\rn)$, $(R_k f)(\t)$ is finite for almost all $\t\in \Pi_{n,k}$ provided that $1\le p<n/k$, and this condition is sharp. However,  $\tilde \Pi_{n,k}$  has  measure zero in $\Pi_{n,k}$. Thus,  what can one say about the existence of the restricted transform
\be (\tilde R_k f)(\t)=(R_k f)(\t)|_{\t\in \tilde\Pi_{n,k}}\ee on functions $f\in L^p (\rn)$?

3. How to describe the range  $\tilde R_k (X)$ where $X=S(\rn)$ is  the Schwartz space of rapidly decreasing smooth functions or any other reasonable function space?

Similar questions can be  posed in other contexts of integral geometry, for instance,   in the elliptic or   hyperbolic space.

All these questions are  in the spirit of \cite {Gelf60} and related  publications, however, the  settings  are not identical, e.g., in the part related to the $L^p$ theory. Our tools are also different.
 In Section 2 we
define the admissible complex $\tilde \Pi_{n,k}$ as a set of all $k$-dimensional planes in $\rn$ which are parallel to a fixed  $(k+1)$-dimensional subspace, for instance, a $(k+1)$-dimensional coordinate plane. We establish sharp conditions for the existence of the corresponding restricted $k$-plane  transform $\tilde R_k$ on $L^p$ functions.  The  explicit inversion formulas for $\tilde R_k$ are simple consequences of the known  inversion formulas for the case of hyperplanes of codimension $1$. The case  of smooth functions is also included. Similar questions are studied in Sections 3 and 4 for geodesic Radon transforms on the sphere and the hyperbolic space. Here we use the same idea  adapted for the corresponding group of motions.

 Section 5  conceptually pertains to Section 2. We have put it at the end of the paper in order not to overload the reader with technicalities. Here  the main result is Theorem \ref{657390sw}
 which  describes  the range of the restricted $k$-plane  transform on the space $S(\rn)$ of rapidly decreasing smooth functions. The classical Helgason theorem for $k=n-1$ \cite [p. 5]{He11} stating that the Radon transform is a bijective map from  $S(\rn)$  onto the corresponding space $S_H (S^{n-1} \times \bbr)$, arises as a particular case of our result. The proof Theorem \ref{657390sw} follows the same scheme as in \cite{He11}, but it is more detailed and, probably, simplified (here we employ some ideas from  Carton-Lebrun \cite{CL}). Moreover,
  our proof gives not only the bijectivity,  but also the continuity of $\tilde R_k$ and its inverse in the respective topologies.

  It is worth mentioning that in the case
  of the {\it overdetermined} $k$-plane transform on  $S(\rn)$,
  different range characterizations   can be found in \cite{Go91, Go10, Gr85, Ku91, Pe91,  Pe92, Pe93, Ri86}.

 It might be of interest to describe the ranges of the restricted transforms from Sections 3 and 4  by making use of the relevant tools of harmonic analysis. This topic can be addressed in  future publications.

\vskip 0.2 truecm

{\bf Some notation.} The following notation will be used throughout the paper. We fix an orthonormal basis $e_1, \ldots, e_n$   in $\bbr^{n+1}$; $S^n$ is the $n$-dimensional unit sphere in $\bbr^{n+1}$. If $\th\in S^n$ is the variable of integration, then $d\th$ stands for the 
$O(n+1)$-invariant measure
 on  $S^n$ and $\sig_n=\int_{S^n}\,d\th=2\pi^{n+1}/\Gam((n+1)/2)$ is the surface area of $S^n$. We write
$d_*\th=\sig_n^{-1} d\th$ for the corresponding normalized measure.

\section{The $k$-plane transform on $\bbr^n$}\label {222222}

\subsection{Definitions}
 In this section we denote
\be\label {hfos4609} \bbr^{k+1}=\bbr e_1 \oplus \cdots \oplus  \bbr e_{k+1}, \qquad \bbr^{n-k-1}=\bbr e_{k+2}\oplus \cdots \oplus  \bbr e_{n}.\ee
 Let  $\tilde \Pi_{n,k}$ be the manifold of all $k$-planes in $\rn$ which are parallel to  $\bbr^{k+1}$. We write $x\in \rn$ as $x=(x', x'')$ where $x'\in \bbr^{k+1}$, $x''\in \bbr^{n-k-1}$.
 Every plane  $\t\in \tilde \Pi_{n,k}$ is parametrized by the triple
 $(\th, s; x'')\in  S^k \times \bbr \times \bbr^{n-k-1}$,
 where $S^k$ is the unit sphere in $\bbr^{k+1}$. Specifically,
\[ \t\equiv \t (\th, s; x'')=\t_0 +x'', \quad \t_0=\{x'\in \bbr^{k+1}: \th \cdot x'=s\}.\]
We  denote $\tilde Z_{n,k}= S^k \times \bbr \times \bbr^{n-k-1}$ and equip this set   with the measure $\tilde d \t= d_*\th ds dx''$.
 Clearly, $\dim \tilde \Pi_{n,k}=n$ and
\be \t (\th, s; x'')=\t (-\th, -s; x'')\quad \forall \;(\th, s; x'')\in \tilde Z_{n,k}.\ee

The $k$-plane transform (\ref{i9435rf7}) restricted to $\tilde\Pi_{n,k}$  has the form
 \be \label{kkmm4539a1}(\tilde R_k f)(\th, s; x'')=\intl_{\th^\perp \cap \bbr^{k+1}} f(s\th +u, x'')\, d_\th u, \ee
where $d_\th u$ is the  volume element of $\th^\perp \cap \bbr^{k+1}$. We shall also write
 \be \label{kkmm4539a1bu}(\tilde R_k f)(\th, s; x'')=(R f_{x''})(\th, s), \quad f_{x''}(\cdot)=f(\cdot, x''),\ee
where $R$ is the usual hyperplane Radon transform in $\bbr^{k+1}$ of a function $f_{x''}(\cdot)$;
 cf. \cite{He11}. Thus,  $\tilde R_k f$ is actually a ``partial'' Radon transform of $f$ in the $x'$-variable, so that  many properties of $R$ can be transferred to  $\tilde R_k$. Below we review some of them.

 \subsection{Existence on $L^p$-functions}


 \begin{theorem} \label{kmmop3z} The integral $\tilde R_k f$ is finite a.e. on $\tilde Z_{n,k}$ if $f$ is locally integrable on $\rn \setminus \{x: x'=0\}$ and
\be  \label{kknn4539a2}\intl_{|x''|<a}dx''\intl_{|x'|>1} \frac{|f(x',x'')|}{|x'|}\, dx'<\infty\quad  \mbox{ for any $\;a>0$ }.\ee
\end{theorem}
This statement follows immediately  from \cite[Theorem  3.2]{Ru13a}.

 \begin{corollary} \label{kmmop3} If $f\in L^p (\rn)$,   $1\!\le\! p\!<\!(k\!+\!1)/k$, then $(\tilde R_k f)(\t)$ is finite for almost all planes   $\t\in \tilde \Pi_{n,k}$. If $p\ge(k\!+\!1)/k$, then there is a function $f_0\in L^p (\rn)$ for which $(\tilde R_k f_0)(\t)\equiv \infty$ on $\tilde \Pi_{n,k}$.
\end{corollary}
 \begin{proof}
The first statement follows from (\ref{kknn4539a2}) by H\"older's inequality. For the second statement we can take, e.g.,
\be  f_0 (x)\!=\!\frac{(2+|x'|)^{-(k+1)/p}\,  e^{-|x''|^2}}{\log^{1/p+\del} (2+|x'|)}, \quad 0\!<\!\del\! <\!1/p', \quad 1/p \!+\!1/p'\!=\!1.\ee
\end{proof}
\noindent {\bf Open problem.} The  condition $p<(k\!+\!1)/k$ differs from $p<n/k$ for the nonrestricted $k$-plane transform. It would be interesting to investigate  restrictions on $p$ for other known admissible complexes; see references  in Introduction.

\subsection{Inversion formulas}

Let $1\le p<(k+1)/k$.
Suppose that
\be  \label{kknn4539a3}\intl_{|x''|<a}\!\!\!dx''\!\!\intl_{\bbr^{k+1}}\!\! |f(x',x'')|^p\, dx'\!<\!\infty \quad  \mbox{ for all $\;a>0$ }.\ee
Then  $\vp=\tilde R_k f$  is well-defined by Theorem \ref{kmmop3z} and the function $f_{x''} (x')\equiv f(x',x'')$ belongs to $L^p (\bbr^{k+1})$ for almost all $x''\in \bbr^{n-k-1}$. Furthermore,  the function $ \vp_{x''} (\th, s)\equiv \vp (\th, s; x'')$ has the form
\be \label{durtkkaz} \vp_{x''} (\th, s)= (Rf_{x''})(\th, s)\ee
where $R$ is the usual hyperplane Radon transform in $\bbr^{k+1}$. Hence, inverting $R$ by any known method
(see, e.g., \cite{Ru04b, Ru13a}), we reconstruct
  $f_{x''} (x')\equiv f(x)$. For example, if $k=1$ and  $f\in L^p (\rn)$,  $1\le p\le 2$, then, by the formula (5.12) from \cite{Ru04b} we have
      \be\label{artemza11kk} f(x)=\frac{1}{\pi}
 \intl_0^\infty \frac{(R^*\vp_{x''}) (x') - (R^*_t \vp_{x''}) (x')}{t^2} \,dt. \ee
 Here $x'\in \bbr^2$, $x''\in \bbr^{n-2}$,
 \[
(R^*_t\vp_{x''}) (x')= \intl_{S^k} \vp_{x''}(\theta, x'\cdot \theta+t)\,d_*\theta, \quad (R^*\vp_{x''})(x')=(R^*_t\vp_{x''})(x')\big |_{t=0}.\]
The integral   $\int_0^\infty (\cdot )$ is understood as  $\lim\limits_{\e\to 0}\int_\e^\infty (\cdot )$ for almost all $x=(x',x'')$ in $\rn$ or in the $L^p (\bbr^2)$-norm for almost all $x'' \in \bbr^{n-2}$.

\section{The Elliptic Case}

\subsection{Definitions}
The $n$-dimensional elliptic space can be interpreted as the  $n$-dimensional unit sphere $S^n$ with identified antipodal points. According to this interpretation, we assume all functions on $S^n$ to be even. 
The Funk transform  $\F $ integrates  a function $f$ on  $S^n$  over $(n\!-\!1)$-dimensional totally geodesic submanifolds of $S^n$ (great circles, if $n=2$). For any $1\le k \le n-1$, the corresponding transform  $\F_k $ is similarly defined by integration  over $k$-dimensional totally geodesic submanifolds  of $S^n$. It can be realized as an integral
\be\label {hfos4609bu0}
(\F_k f)(\xi)=\intl_{S^n \cap \xi} f(\th)\, d_\xi \th, \qquad \xi \in G_{n+1, k+1},
\ee
where $G_{n+1, k+1}$ is the Grassmann manifold of $(k+1)$-dimensional linear subspaces of $\bbr^{n+1}$ and  $d_\xi \th$ denotes the $O(n+1)$-invariant probability measure on $S^n \cap \xi$. 

Suppose $k<n-1$. Then
 $\dim G_{n+1, k+1}= (k+1)(n-k)>n$ and the inversion problem for $\F_k f$ is overdetermined. Our aim is to define an $n$-dimensional admissible complex $\tilde G_n$ in $G_{n+1, k+1}$, so that $f$ could be explicitly reconstructed from $(\F_k f)(\xi)$, $\xi \in \tilde G_n$.
 We will be using the same idea as in the Euclidean case, but translations will be  replaced by orthogonal transformations. Let
\be\label {ttt609bu0cc} \bbr^{n-k}=\bbr e_1 \oplus \cdots \oplus  \bbr e_{n-k}, \qquad \bbr^{k+1}=\bbr e_{n-k+1}\oplus \cdots \oplus  \bbr e_{n+1},\ee
 \[  S^{n-k-1}=S^n \cap \bbr^{n-k}.\]
Fix a point $ v\in S^{n-k-1}$ and denote
\be\label {hfos4609bu2xy1}   \bbr^{k+2}_v=\bbr v \oplus \bbr^{k+1}, \qquad S^{k+1}_v=S^n \cap \bbr^{k+2}_v.\ee
Clearly, every point $\th = (\th_1, \ldots, \th_{n+1}) \in S^n$ belongs to some $S^{k+1}_v$. Specifically, if $\th' = (\th_1, \ldots, \th_{n-k})\neq 0$, then $v=\th'/|\th'|$. If $\th' =0$ then $\th \in S^{k+1}_v$ for all $v\in S^{n-k-1}$. We define
\be\label {hfos4609bu2x}
 \tilde G_n=\{\xi\in G_{n+1, k+1}: \xi\subset \bbr^{k+2}_v \quad \mbox{\rm for some $ v\in S^{n-k-1}$}\},\ee
 which is a fiber bundle over $S^{n-k-1}$ with fibers isomorphic to the Grassmannian $G_{k+2, k+1}$. It will be shown that $\tilde G_n$ is an admissible complex for $\F_k $.
 
  Clearly, $\dim \tilde G_n=n$.    Furthermore, every $\xi\in \tilde G_n$ can be parametrized as $\xi=\xi (v,w)$ where $v\in S^{n-k-1}$, $w\in S^{k+1}_v$, $w\perp \xi$, 
   so that $\xi (v,w)=\xi (\pm v, \pm w)$ with any combination of pluses and minuses. We denote
 \be\label {hfos4609bu33} \tilde S_n =\{ (v,w): \;v\in S^{n-k-1}, \, w\in S^{k+1}_v\}.\ee
 and equip $\tilde S_n $ with the product measure $d_*v d_*w$ where
 $d_*v$ and $d_*w$ stand for the corresponding  probability measures on $S^{n-k-1}$ and    $S^{k+1}_v$, respectively.
  The  transformation (\ref{hfos4609bu0}) restricted to $\tilde G_n$ can be realized as
\be\label {hfos4609bu4}
 (\tilde \F_k f)(v,w)=\intl_{\{\th \in  S^{k+1}_v: \; \th \cdot w=0\}} f(\th) \,d_{v,w} \th, \qquad (v,w)\in \tilde S_n, \ee
which is the usual Funk transform on  $S^{k+1}_v$ with
 the relevant normalized surface measure $d_{v,w} \th$.  Clearly,
$(\tilde \F_k f)(v,w) = (\tilde \F_k f)(\pm v, \pm w)$.

 \subsection{Existence on $L^p$-functions}
  Fix  any $v\in S^{n-k-1}$ and choose an orthogonal transformation $\gam_v$ in the coordinate plane  $\bbr^{n-k}$,  so that $\gam_v  e_{n-k}=v$. Let
\be\label {hfos4609bu5a}
\tilde \gam_v =\left[\begin{array}{ll}  \gam_v &0\\
0& I_{k+1}
\end{array}\right] \in O(n+1).\ee
Then
\be\label {hoopmvru4}
(\tilde \F_k f)(v,\tilde \gam_v \zeta)=(\F f_v)(\zeta), \qquad  f_v(\eta)=f (\tilde \gam_v \eta),\ee
where
\be\label {hoopmvru9}  (\F f_v)(\zeta)\!\equiv\!\intl_{ \eta \cdot\zeta=0}   f_v(\eta)\, d_\zeta \eta \ee
is the usual Funk transform on the sphere $S^{k+1}$ defined by
\be\label {hoo34d4pp}
S^{k+1}=S^n \cap \bbr^{k+2},  \qquad \bbr^{k+2}=\bbr e_{n-k}\oplus \cdots \oplus  \bbr e_{n+1}.\ee
 Thus, the existence of $\tilde \F_k f$ is equivalent to the existence of the Funk transform  (\ref{hoopmvru9}). The latter is well-defined whenever $f_v\in L^1 (S^{k+1})$ and may not exist otherwise (take, e.g., $f_v(\eta) =|\eta_{n+1}|^{-1}\notin L^1 (S^{k+1})$, for which $(\F f_v)(\zeta)\equiv \infty$; cf. \cite[formula (2.12)]{Ru02b}). This observation yields the following.
 \begin{theorem} Let   $1\le k\le n-1$,
\be  \label {hfuuuu82} \intl_{S^n} |f(\th)|\, \frac{ d_*\th}{|\th'|^{n-k-1}}<\infty,\ee
where  $\th'=(\th_1, \ldots, \th_{n-k})$ is the orthogonal projection of $\th$ onto the coordinate plane $\bbr^{n-k}=\bbr e_1 \oplus \cdots \oplus  \bbr e_{n-k}$.
Then
\be\label {hfos4609bu82}
\intl_{\tilde S_n} (\tilde \F_k f)(v,w)\, d_*v d_*w=\frac{2\, \sig_{n}}{\sig_{k+1}\sig_{n-k-1}}\,\intl_{S^n} f(\th)\, \frac{ d_*\th}{|\th'|^{n-k-1}}.\ee
\end{theorem}
\begin{proof} If $k=n-1$, then  (\ref{hfos4609bu82}) is  a particular case of the duality for Radon-type transforms \cite{He11}. Denote the left-hand side of (\ref{hfos4609bu82}) by $I$. Since
\[
\intl_{S^{k+1}} (\F f_v)(\zeta) \, d_*\z =\intl_{S^{k+1}} f_v(\eta)\,d_*\eta, \qquad  f_v(\eta)=f (\tilde \gam_v \eta),\]
then, by (\ref{hoopmvru4}),
\bea &&I\!=\!\!\intl_{S^{n-k-1}} \!\!\!d_*v \!\intl_{S^{k+1}_v}(\tilde \F_k f)(v,w) d_*w\!=\!\!\!
\intl_{S^{n-k-1}} \!\!\!\!d_*v\!\!\intl_{S^{k+1}} \!(\F f_v)(\zeta) \, d_*\z\nonumber\\
&&\label {hfos466682}  \quad  \!=\!\intl_{S^{n-k-1}}\!\!\! d_*v\intl_{S^{k+1}} f_v(\eta)\,d_*\eta=\frac{1}{\sig_{k+1}}\intl_{S^{n-k-1}} d_*v\intl_{S^{k+1}}f (\tilde \gam_v \eta)\,d\eta\\
&&\quad \!=\!\frac{1}{\sig_{k+1}}\intl_{S^{n-k-1}} \!\!\!d_*v\intl_0^\pi \sin^k\psi\, d\psi\intl_{S^{k}} f(v\cos \psi + \om\sin \psi)\, d\om.\nonumber\eea
Using the bi-spherical coordinates \cite [pp. 12, 22]{VK}
\be \th= v\cos \psi + \om\sin \psi, \quad d\th=\sin^k\psi\,\cos^{n-k-1} \psi\, d\psi dv d\om,\ee
\[  v\in S^{n-k-1}\subset \bbr^{n-k}, \qquad \om \in S^k\subset \bbr^{k+1}, \qquad  0\le \psi\le \pi/2,\]
 and noting that $\cos \psi=|\th'|$, we continue:
\bea I&=&\frac{2}{\sig_{k+1}\sig_{n-k-1}}\intl_0^{\pi/2} \sin^k\psi\, d\psi\intl_{S^{n-k-1}} dv\intl_{S^{k}} f(v\cos \psi + \om\sin \psi)\, d\om\nonumber\\
&=&\frac{2\, \sig_{n}}{\sig_{k+1}\sig_{n-k-1}}\intl_{S^n} f(\th)\, \frac{ d_*\th}{|\th'|^{n-k-1}},\nonumber\eea
as desired.
\end{proof}
 \begin{corollary}  Let   $1\le k< n-1$, $f\in L^p(S^n)$, $n-k<p\le \infty$. Then $(\tilde \F_k f)(v,w)$ is finite for almost all $(v,w)\in \tilde S_n$ and the operator $\tilde \F_k$ is bounded from $L^p(S^n)$ to $L^1 (\tilde S_n)$. If $p\le n-k$, then there is a function  $\tilde f\in L^p(S^n)$ for which $(\tilde \F_k \tilde f)(v,w)=\infty$. Specifically, \[\tilde f(\th)=|\th'|^{-1} (1-\log |\th'|)^{-1}, \qquad  \th'=(\th_1, \ldots, \th_{n-k}).\]
\end{corollary}
 \begin{proof} The first statement follows from (\ref{hfos4609bu82}) by H\"older's inequality, because the integral
 \[\intl_{S^n}\frac{ d_*\th}{|\th'|^{(n-k-1)p'}}=\sig_{k}\,\sig_{n-k-1}\intl_0^{\pi/2} \sin^k\psi\,\cos^{(n-k-1)(1-p')} \psi\, d\psi, \quad \frac{1}{p}+ \frac{1}{p'}=1,\]
 is finite if $n-k<p$. To prove the second statement, we have,
 \bea\intl_{S^n} |\tilde f(\th)|^p\,d\th&=&\sig_{k}\,\sig_{n-k-1}\intl_0^{\pi/2}  (1-\log (\cos \psi))^{-p} \sin^k\psi\,\cos^{n-k-1-p} \psi\, d\psi\nonumber\\
 &=&\sig_{k}\,\sig_{n-k-1}\intl_0^1  (1-\log s)^{-p} (1-s^2)^{(k-1)/2} \,s^{n-k-1-p} \, ds.\nonumber\eea
 This integral is finite  if $p\le n-k$, $1\le k< n-1$.

 Let us show that $(\tilde \F_k \tilde f)(v,w)=\infty$. By (\ref{hoopmvru4}), it suffices to prove that $(\F \tilde f_v)(\z)=\infty$. For $\eta =(\eta_{n-k}, \ldots, \eta_{n+1})\in S^{k+1}\subset  \bbr^{k+2}$ (see (\ref {hoo34d4pp})), we have $\tilde f_v(\eta)= |\eta_{n-k}|^{-1}\,(1-\log |\eta_{n-k}|)^{-1}$.
  Let $r_\z$ be an orthogonal transformation in  $\bbr^{k+2}$ such that $r_\z e_{n-k}=\z$, and let $S^k$ be  the unit sphere in the plane $\bbr^{k+1}=\bbr e_{n-k+1}\oplus \cdots \oplus  \bbr e_{n+1}$. Setting $\eta=r_\z u$, we have
 \bea (\F \tilde f_v)(\z)&=&\intl_{ \eta \cdot\zeta=0}  |\eta_{n-k}|^{-1}\,(1-\log |\eta_{n-k}|)^{-1} \,d_\zeta \eta\nonumber\\
&=&\intl_{S^k} |r_\z u \cdot e_{n-k}|^{-1}\, (1-\log |r_\z u \cdot e_{n-k}|)^{-1}\, du,\nonumber\eea
where
\[
|r_\z u \cdot e_{n-k}|=|u \cdot r_\z^{-1}e_{n-k}|=|u \cdot \Pr_{\bbr^{k+1}} (r_\z^{-1}e_{n-k})|= h\, |u\cdot \sig|,\]
\[ h=|\Pr_{\bbr^{k+1}} (r_\z^{-1}e_{n-k})|, \quad  \sig=\Pr_{\bbr^{k+1}} (r_\z^{-1}e_{n-k})/h \in S^k. \]
Hence,
\bea
&&(\F \tilde f_v)(\z)=\intl_{S^k}  (h\, |u\cdot \sig|)^{-1}\, (1-\log (h\, |u\cdot \sig|))^{-1}\, du\nonumber\\
&&=\frac{2\sig_{k-1}}{h}\intl_0^1 \frac{(1-t^2)^{k/2 -1}}{ t\,(1-\log (th))}\,dt\ge \frac{c_k}{h} \intl_0^{1/2} \frac{dt}{t\,(1-\log (th))},\nonumber\eea
$c_k=\const$. The last integral diverges.
\end{proof}

\subsection{Inversion formulas}
To reconstruct $f$  from
$\tilde \F_k f$, it suffices to invert the usual Funk transform $\F$ in (\ref{hoopmvru4}) by any known method; see, e.g., \cite{GGG03, He11, Pa04, Ru98a, Ru02a, Ru02b, Ru13a}. Let $\vp(v,w)= (\tilde \F_k f)(v,w)$, $\vp_v(\z)=\vp(v,\tilde \gam_v \z)$, where $\tilde \gam_v$ has the form  (\ref{hfos4609bu5a}). Then
\be\label {hfos4609bu6} f_v(\eta) \equiv f(\tilde \gam_v \eta)= (\F^{-1} \vp_v)(\eta).\ee
If $f$ is a continuous function, its value
 at a point $\th \in S^n$ can be found as follows. Interpret  $\th$ as a column vector
$\th=(\th_1, \ldots, \th_{n+1})^T$ and set
\[\th'=(\th_1, \ldots, \th_{n-k})^T\in \bbr^{n-k},\qquad \th''=(\th_{n-k+1}, \ldots, \th_{n+1})^T\in \bbr^{k+1},\]
\be\label {hfos4609bu71}  v=\th'/|\th'|\in S^{n-k-1}\subset\bbr^{n-k}, \ee
\be\label {hfos4609bu72} \eta=(0, \ldots, 0,|\th'|, \th'')^T\in S^{k+1}\subset \bbr^{k+2}.\ee
If $\th'\neq 0$, then $\tilde \gam_v \eta=\th$ and we get
\be\label {hfos4609bu7} f(\theta)= (\F^{-1} \vp_v)(\eta),\qquad \vp_v(\zeta)=\vp (v, \tilde \gam_v \zeta).\ee
If $\th'=0$, then $f(\theta)$ can be reconstructed by continuity.

 If $f$ is an arbitrary  function  satisfying (\ref{hfuuuu82}), for instance,  $f\in L^p(S^n)$, $n-k<p\le \infty$, then the integral
\[\intl_{S^{n-k-1}}\!\!\! d_*v\intl_{S^{k+1}} |f_v(\eta)|\,d_*\eta\]
is finite; see calculations after (\ref{hfos466682}). Hence, by (\ref{hfos4609bu6}), $f$ can be explicitly reconstructed at almost all points on almost all spheres $S^{k+1}_v$ by making use of known inversion formulas for the Funk transform on this class of functions; see, e.g., \cite{Ru02a, Ru02b, Ru13a}.

\section{The Hyperbolic Case}

The hyperbolic totally geodesic Radon transform  assigns to each sufficiently good function
$f$ on the real hyperbolic space $\bbh^n$ the collection of integrals of $f$ over $k$-dimensional
totally geodesic submanifolds of $\bbh^n$. If $k<n-1$, the inversion problem for this transform is overdetermined and we shall construct the corresponding admissible complex, using the same idea as in the spherical  case. Realization of this idea relies on the geometry of $\bbh^n$.

\subsection{Preliminaries}

We recall basic facts. More details can be found in  \cite{BR99a, BR99b, Ru02c, VK}; see also \cite {GGG03, GGV}. Let $\bbe^{n, 1}\sim \bbr^{n+1}, \; n \ge 2$, be the real pseudo-Euclidean space
of points $x = (x_1, \dots, x_{n+1})$ with the inner product
\be\label {hfohhh}
[x, y] = -x_1y_1-\dots - x_ny_n+x_{n+1} y_{n+1}.\ee  The
$n$-dimensional real hyperbolic
space $\bbh^n$ will be realized as the ``upper" sheet of the two-sheeted
hyperboloid \be \bbh^n = \{ x \in \bbe^{n, 1}: [x, x] = 1,
x_{n+1} > 0\};\ee
dist$(x,y)=\cosh^{-1}[x,y]$ is the geodesic distance between  the points $x$ and $y$ in $\bbh^n$. The
hyperbolic coordinates $\theta_1, \ldots, \theta_{n -1}, r $ of a point
$x\in \bbh^n$ are defined by
\be\label {hfohhh1tag 2.3}  \left \{\begin{array} {l} x_1 = \text{sinh} r \sin \theta_{n -1} \ldots
\sin \theta_2 \sin \theta_1, \\
x_2 = \text{sinh} r \sin \theta_{n -1} \ldots \sin \theta_2 \cos \theta_1, \\
......................................................\\
x_n = \text{sinh} r \cos \theta_{n -1}, \\
x_{n +1} = \text{cosh} r,
\end{array}\right.
\ee
where $0 \le \theta_1 < 2 \pi; \ 0 \le \theta_j < \pi, \ 1 < j \le n - 1; \ 0 \le r < \infty$.
 Every $x \in \bbh^n$ is representable as
\be\label {hfrry3}  x = \theta\,\sinh \,r  +  e_{n+1}\,\cosh \,r \ee where
$\theta$ is a point of the unit sphere $S^{n -1}$ in $\bbr^n=\bbr e_1 \oplus \cdots \oplus \bbr e_{n}$ with
Euler's angles  $\theta_1, \ldots, \theta_{n -1}$.
 We denote by $SO_0 (n, 1)$  the  connected group  of pseudo-rotations
of $E^{n,1}$ which preserve the bilinear form  (\ref{hfohhh});  $SO (n)$ is the
rotation group in $\rn$ which is identified with  the subgroup of all pseudo-rotations leaving $e_{n +1}$ fixed.
 The $SO (n)$-invariant measure $d \theta$ on $S^{n -1}$ is  defined by
\be\label {hfohhh1tag 2.4} d \theta \!=\! J (\theta ) \ d \theta_1 \ldots d \theta_{n -1}, \quad J (\theta) \!=\! \sin^{n -2} \theta_{n -1}
\sin^{n - 3} \theta_{n - 2} \ldots \sin \theta_2, \ee
 so that
$ \sigma_{n -1} = \int_{S^{n -1}} \,d \theta = 2 \pi^{n/2} / \Gamma (n/2)$.
The $SO_0 (n, 1)$-invariant measure $dx $ on $\bbh^n$  can be defined     by
\be\label {hfohhh1tag 2.7}  d x = \text{sinh}^{n -1} r \ d \theta d r \ee
 where
 $d \theta$ has the form (\ref{hfohhh1tag 2.4}). Then
 \be\label {hfrr2a3}\intl_{\bbh^n} f(x)\, dx=\intl_0^\infty \sinh^{n -1} r \,d r\intl_{S^{n -1}} f(\theta\,\sinh \,r  +  e_{n+1}\,\cosh\, r)\,d \theta.\ee
 We will need a generalization of (\ref{hfrry3}) of the form
\be\label {hfohhh2}  x =
v\,\sinh\,r +u \,\cosh\,r,\ee
\[ v\in
S^{n-k-1}\!\subset\!\bbr^{n-k}, \quad u \!\in\!\bbh^k\!\subset\!\bbe^{k, 1} \!\sim \!\bbr^{k+1},\quad 0\! \le\! r\!<\!\infty, \quad 0\!\le \! k\!<\!n,\]
where
 $ \bbr^{n-k}$ and $ \bbr^{k+1}$ have the same meaning as in (\ref{ttt609bu0cc}). Then
\be\label {hfohhh3} dx= dv \,du\, d\nu(r), \quad
d\nu(r)=\sinh^{n-k-1}\, r\,\cosh^k\, r \, dr,\ee
$dv$ and $du$ being the Riemannian measures on
$S^{n-k-1}$  and $\bbh^k$, respectively; see \cite [pp. 12, 23]{VK}.  Owing to (\ref{hfohhh2}),
\be\label {hfohhh4}
 \intl_{\bbh^n} f(x)\,dx=\intl^\infty_0 d\nu(r)\intl_{S^{n-k-1}}dv \intl_{\bbh^k}
 f(v\,\sinh\,r +u \,\cosh\,r) \, du.\ee
The case $k=0$  agrees with  (\ref{hfrr2a3}). If $k=n-1$, then (\ref{hfohhh4}) yields
\be\label {hfoeehhh4}
 \intl_{\bbh^n} f(x)\,dx=\intl^\infty_{-\infty} \cosh^{n -1} r \,d r\intl_{\bbh^{n-1}}
f(e_1\,\sinh\,r  + u\,\cosh\,r) \, du.\ee

 We will also need the one-sheeted
hyperboloid
\be\label {hfohhh1}\overset {*}{\bbh}{}^n = \{ y \in E^{n, 1} : [y, y]= - 1 \}.\ee
Every point
$y \in \overset {*}{\bbh}{}^n$ is  representable as
$y =  \sigma\, \cosh \,\rho +e_{n+1}\,\sinh \,\rho$ where $- \infty < \rho < \infty$ and
$\sigma \in S^{n -1}$.
In this notation the $SO_0 (n, 1)$-invariant measure $dy$ on $\overset {*}{\bbh}{}^n$ has the form
\be\label {hfohhh1tag 2.11} dy = \text{cosh}^{n -1} \rho \ d \sigma d \rho. \ee

The hyperbolic  Radon
transform of a sufficiently good function $f$ on $\bbh^n$ is defined as a function on  $\overset {*}{\bbh}{}^n$ by the formula
\be\label{hyptag31} (\H f) (y)  = \int\limits_{\xi_y}
f (x) \! \ d_y x, \quad \xi_y \!= \!\{ x \!\in\!
\bbh^n : [ x, y] \!=\! 0 \}, \quad y\in \overset {*}{\bbh}{}^n,\ee
where the measure
$d_y x$  is the image of the standard measure on the $(n-1)$-dimensional hyperboloid $\xi_{e_n}=\{x\in \bbh^n: x_n=0\}$  under the transformation
$\omega_y \in SO_0(n, 1)$ satisfying $\omega_y \xi_{e_n} = \xi_y$.

\begin{theorem}\label{hyptag31th} \cite[Corollaries 3.7, 3.8]{BR99a} The integral (\ref{hyptag31}) is
 finite for almost all $y \in \overset {*}{\bbh}{}^n$ whenever
\be\label{hyptag31f}
\intl_{\bbh^{n}}|f(x)| \, \frac{dx}{x_{n+1}}  <\infty.\ee
In particular, (\ref{hyptag31}) is finite a.e. if $f \in L^p (\bbh^n)$,
\be\label{hyptavvg31f} 1 \le p < (n -1)/ (n -2).\ee The condition (\ref{hyptavvg31f}) is sharp.   Moreover, for every $\sigma \in S^{n -1}$,
\be\label{hyptag31fa1}
\intl_{-\infty}^\infty (\H f) (\sigma\, \cosh \,\rho +e_{n+1}\,\sinh \,\rho)\, \frac{d\rho}{\cosh \,\rho}=\intl_{\bbh^{n}} f(x) \, \frac{dx}{x_{n+1}}\ee
and this expression does not exceed $c\, ||f||_p$, $c=\const$.
  \end{theorem}

  Explicit inversion formulas for $\H f$ and other properties of this transform can be found in \cite {BC91, BC92, BC96, BR99a, He11, Ru02a, Ru02c, Str81}.

\subsection{Radon transform over $k$-geodesics in $\bbh^n$}

 We denote by  $\Xi_k$ the set of all $k$-dimensional
totally geodesic submanifolds $\xi$ of $ \bbh^n$, $ 1 \le k \le
n-1$.  The corresponding Radon transform  has the form
\be\label {hfohhh5} (\H_k f)(\xi)=\intl_{\{x\in  \bbh^n:\, d(x, \xi)=0\}} f(x) \, d_\xi x, \qquad \xi \in \Xi_k,\ee where $d_\xi x$ is the
 volume element on $\xi$. To give this formula  precise meaning, we set
 $\bbe^{n, 1} = \bbr^{n-k}\times\bbe^{k,1}$ where
\be\label {hfr609b}
\bbr^{n-k} \!=\! \bbr e_1 \oplus \ldots \oplus \bbr e_{n-k},\quad
\bbe^{k,1} \! \sim \!\bbr^{k+1}\!=\! \bbr e_{n-k+1} \oplus \ldots \oplus
\bbr e_{n+1},
\ee
and  let   $ r_\xi \in SO_0 (n, 1)$ be a pseudo-rotation which takes the $k$-dimensional hyperboloid $\bbh^k= \bbh^n\cap\bbr^{k+1}$ to $\xi$.  Then
\be\label {hfohhh5}
 (\H_k f)(\xi)=\intl_{\bbh^k} f(r_\xi x)\,d_{\bbh^k} x\ee
where the measure $d_{\bbh^k} x$ on $\bbh^k$ is defined in a standard way. If $f\in L^p(\bbh^n)$,  this integral is finite for almost all $\xi\in \Xi_k$  provided that
\be\label {hfohhh5fr} 1\le p <(n-1)/(k-1),\ee
 and  this condition is sharp, \cite{BR99b, Str81}.

Suppose $k<n-1$ and define an $n$-dimensional admissible complex $\tilde \Xi_n$ in $\Xi_k$,  so that $f$ can be explicitly reconstructed from $(\H_k f)(\xi)$, $\xi \in \tilde \Xi_n$. We fix a point $v\in S^{n-k-1}\subset \bbr^{n-k}$ (see (\ref{hfr609b})) and denote
\[
\bbr^{k+2}_v \!=\! \bbr v \oplus \bbr^{k+1}, \qquad  \bbh^{k+1}_v=\bbh^{n} \cap \bbr^{k+2}_v;\]
cf. (\ref{hfos4609bu2xy1}). Then we set
\be\label {hfr609b1}  \tilde \Xi_n =\{\xi\in\Xi_k: \xi\subset \bbh^{k+1}_v \quad \mbox{\rm for some } \; v\in S^{n-k-1}\};\ee
cf. (\ref{hfos4609bu2x}). Clearly, $\dim \tilde \Xi_n=n$.  Every $k$-geodesic $\xi\in\tilde \Xi_n$  can be indexed  by the pair  $(v,y)$ where $v\in S^{n-k-1}$ and $y$ is a point of the one-sheeted hyperboloid  $\overset {*}{\bbh}{}^{k+1}_v\!= \!\overset {*}{\bbh}{}^{n}\cap \bbr^{k+2}_v$.   We denote
\be\label {hfmmu3mu}
 \tilde\bbh_n=\{ (v,w): v\in S^{n-k-1}, \; w\in \overset {*}{\bbh}{}^{k+1}_v\}\ee
  and equip this set with the product measure $dvdw$, where $dv$ and $dw$  are canonical measures on  $S^{n-k-1}$ and $\overset {*}{\bbh}{}^{k+1}_v$, respectively; cf. (\ref{hfohhh1tag 2.4}), (\ref{hfohhh1tag 2.11}).    The Radon transform (\ref{hfohhh5}) restricted to $\tilde \Xi_n$ can be realized as
\be\label {hfos4609bu4}
 (\tilde \H_k f)(v,w)=\intl_{\{x\in  \bbh^{k+1}_v:\, [x,w]=0\}} f(x) \,d_{v,w} x, \qquad (v,w)\in \tilde \bbh_n, \ee
 where the measure $d_{v,w} x$  is defined as the image of the corresponding canonical measure on $\bbh^k$. Clearly, (\ref{hfos4609bu4}) is the usual hyperbolic Radon transform on the $(k+1)$-dimensional hyperboloid $\bbh^{k+1}_v$; cf. (\ref{hyptag31}).

 \subsection{Existence on $L^p$-functions}
  Fix   $v\in S^{n-k-1}$ and choose  an arbitrary   rotation $\gam_v$ in   $\bbr^{n-k}$,  so that $\gam_v  e_{n-k}=v$. As in (\ref{hfos4609bu5a}), let
\be\label {hfcdcfa}
\tilde \gam_v =\left[\begin{array}{ll}  \gam_v &0\\
0&I_{k+1}
\end{array}\right]\ee
and change variables in (\ref{hfos4609bu4}) by setting
 $x=\tilde \gam_v  \eta$, $ w=\tilde \gam_v  \z$.
Then $\eta\in \bbh^{k+1}$, $\, \zeta\in \overset {*}{\bbh}{}^{k+1}$,
\be\label {hoo34d4}
\bbh^{k+1}=\bbh^n \cap \bbr^{k+2},  \qquad \bbr^{k+2}=\bbr e_{n-k}\oplus \cdots \oplus  \bbr e_{n+1}.\ee
We get
\be\label {hoyyyru9} (\tilde \H_k f)(v,\tilde \gam_v  \z)=(\H f_v)(\z), \qquad f_v(\eta)=f (\tilde \gam_v \eta),\ee
where
\be\label {hoopmvyyy}  (\H f_v)(\zeta)\!\equiv\!\intl_{\{\eta \in \bbh^{k+1}:\, [\eta, \zeta]=0\}}   f_v(\eta)\, d_\zeta \eta \ee
is the usual hyperbolic Radon transform on $\bbh^{k+1}$. Thus, the existence of $\tilde \H_k f$ is equivalent to the existence of $(\H f_v)(\z)$. The  latter is characterized by Theorem  \ref{hyptag31th} which should be applied to
 $f_v$. To reformulate the conditions of that theorem in terms of $f$, we need the following

\begin{lemma} The equality
\be\label {hoyiii} \intl_{S^{n-k-1}}\!\! \!dv\!\intl_{\bbh^{k+1}}\!\! f_v (\eta)\,d\eta\!=\!2 \intl_{\bbh^n} \frac{f(x)}{|x'|^{n-k-1}}\,dx, \quad x'\!=\!(x_1, \ldots, x_{n-k}),\ee
holds provided that either side of it is finite when $f$ is replaced by $|f|$.
\end{lemma}
\begin{proof}
 Let $\eta=\eta_{n-k} \, e_{n-k}+\tilde \eta$, $\tilde \eta=(\eta_{n-k+1}, \ldots \eta_{n+1})$. Then $\tilde \gam_v \eta=v\, \eta_{n-k}+\tilde \eta$ and (\ref {hfoeehhh4}) yields
\bea l.h.s &=& \intl_{S^{n-k-1}} dv \intl_{\bbh^{k+1}} f(v\, \eta_{n-k}+\tilde \eta)\,d\eta\nonumber\\
 &=&\intl_{S^{n-k-1}} dv \intl_{-\infty}^\infty \cosh^k r\, dr   \intl_{\bbh^{k}}  f(v\,\sinh\, r +u\,\cosh\, r)\,du\nonumber\\
 &=&2\intl_0^\infty \frac{d\nu (r)}{\sinh^{n-k-1} r} \intl_{S^{n-k-1}} dv    \intl_{\bbh^{k}}  f(v\,\sinh\, r +u\,\cosh\, r)\,du,\nonumber\eea
$d\nu(r)=\sinh^{n-k-1}\, r\,\cosh^k\, r \, dr$. By (\ref{hfohhh4}), the result follows.
\end{proof}

 \begin{theorem} Let   $1\le k\le n-1$. The integral (\ref{hfos4609bu4}) is
 finite for almost all $(v,w)\in \tilde \bbh^n$ provided that
\be\label{hyptag31f}
\intl_{\bbh^{n}}|f(x)|^p \,\frac{dx}{|x'|^{n-k-1}} <\infty, \qquad 1 \le p < k/ (k-1).\ee
 \end{theorem}
\begin{proof} If $f$ satisfies (\ref{hyptag31f}), then, by (\ref{hoyiii}), $f_v\in L^p (\bbh^{k+1})$. Hence, by Theorem  \ref{hyptag31th}  and (\ref{hoyyyru9}),
$(\tilde \H_k f)(v,\tilde \gam_v  \z)=(\H f_v)(\z)$ is finite for almost all $v\in S^{n-k-1}$ and $\z\in \overset {*}{\bbh}{}^{k+1}$. It follows that
 $(\tilde \H_k f)(v,w)$ is finite for almost all $v\in S^{n-k-1}$ and $w\in \overset {*}{\bbh}{}^{k+1}_v$.
\end{proof}
\begin{remark} {\rm  The restriction  $1 \le p < k/ (k-1)$ is sharp, as in Theorem  \ref{hyptag31th}, and the bound  $k/ (k-1)$ is smaller than $(n-1)/ (k-1)$ if $k<n-1$; cf. (\ref{hfohhh5fr}).}
\end{remark}

\subsection{Inversion formulas} To reconstruct  an arbitrary  function $f$ satisfying (\ref{hyptag31f}) from
$\vp(v,w)= (\tilde \H_k f)(v,w)$, it suffices to invert the usual hyperbolic Radon transform $(\H f_v)(\z)$ from
(\ref{hoopmvyyy}). Specifically,  fix $v\in S^{n-k-1}$ and let $\vp_v(\z)=\vp(v,\tilde \gam_v \z)$. Using any known inversion formula for $\H$  (see, e.g., \cite{BR99a, He11,  Ru02a, Ru02c}), we get
\be\label {dddwbu71w} f_v(\eta)\equiv f (\tilde \gam_v \eta)=(\H^{-1} \vp_v)(\eta).\ee
Since $f_v\in L^p (\bbh^{k+1})$, $1 \le p < k/ (k-1)$, then it can be  evaluated at almost all points of almost all hyperboloids
$\bbh^{k+1}_v$. If, in addition to (\ref{hyptag31f}), $f$ is continuous, then,
to find the value of $f$ at a  point $x\in \bbh^n$, we regard $x$ as a column vector $x=(x_{1}, \ldots, x_{n+1})^T$ and  set
\[x'=(x_1, \ldots, x_{n-k})^T\in \bbr^{n-k},\qquad x''=(x_{n-k+1}, \ldots, x_{n+1})^T\in \bbr^{k+1},\]
\be\label {dddwbu71}  v=x'/|x'|\in S^{n-k-1}\subset\bbr^{n-k}, \ee
\be\label {dddwbu72} \eta=( 0, \ldots, 0, |x'|, x'')^T\in \bbh^{k+1}\subset \bbr^{k+2}.\ee
Then  $x=\tilde \gam_v \eta$ and we get
$f(x)= (\H^{-1} \vp_v)(\eta)$.

\section{The Range of the Restricted $k$-plane Transform}

\subsection{Definitions and the main result}

We will be using the same notation as in Section \ref {222222}. In the following $\bbz_+=\{0,1,2, \ldots\}$, $\bbz^n_+=\bbz_+ \times \cdots \times \bbz_+ $ ($n$ times).
 The  Schwartz space $S (\rn)$  is defined in a standard way with the topology generated by the sequence of norms
\[ ||f||_m= \sup\limits_{|\a|\le m} (1+|x|)^m |(\partial^\a f)(x)|, \qquad m=0,1,2, \ldots.\]
The Fourier transform of $f\in S (\rn)$ has the form \be \label{ft}(Ff)(y)
\equiv \hat f (y) = \intl_{\bbr^{ n}} f(x)\, e^{ i x \cdot y} \,dx.
\ee
The corresponding inverse Fourier transform  will be denoted by $\check f$.

\begin{definition} \label{iiuuyg5a} A function $g$ on the sphere $S^{k} \subset \bbr^{k+1}$ is called differentiable
if the homogeneous function $\tilde g(x)= g(x/|x|)$ is
differentiable in the usual sense on $\bbr^{k+1} \setminus \{0\}$. The
derivatives of  $g$  will be defined as
restrictions to $S^{k}$ of the corresponding derivatives of
$\tilde g(x)$:
\be\label{iiuuyg}
(\partial_\theta^\a g)(\theta)=(\partial^\a \tilde g)(x)\big|_
{x=\theta},\qquad \a \in \bbz_+^{k+1}, \quad \theta\in S^{k}.
\ee
\end{definition}

\begin{definition}
 We  denote by $S_e (\tilde Z_{n,k})$ the  space of
 functions $\vp (\th, s; x'')$ on $ \tilde Z_{n,k}= S^k \times \bbr \times \bbr^{n-k-1}$,  which are
infinitely differentiable in $\theta$, $s$ and $ x''$, rapidly decreasing
as $|s|+|x''| \to \infty$ together with all derivatives, and satisfy
\be \label {65679z90swu}\vp (-\th, -s; x'') = \vp (\th, s; x'') \qquad \forall \; (\th, s; x'')\in \tilde Z_{n,k}.\ee
 The topology in $S_e (\tilde Z_{n,k})$ is defined by the sequence of norms
\be\label{knnzwe35}
||\vp ||_m=\!\sup\limits_{|\mu|+j+|\gam|\le m} \,\sup\limits_{\theta,s, x''} \
\!(1\!+\!|s|\!+\!|x''|)^m |(\partial_\theta^\mu \partial^j_{s} \partial_{x''}^\gam \vp)(\th, s; x'')|. \ee
The space  $S_e (Z_{n})$  of rapidly decreasing even smooth
 functions $\tilde\vp (\th, s)$ on $Z_{n}= S^{n-1} \times \bbr$ is defined similarly.
 \end{definition}

 \begin{definition} \label{iiuuyg5a1kka2}
Let  $S_H (\tilde Z_{n,k})$ denote the subspace of all functions $\vp \in S_e (\tilde Z_{n,k})$ satisfying the  {\bf moment condition}:
 {\it  For every $m\in \bbz_+$ there exists a homogeneous polynomial
\[P_{m} (\th, x'')=\sum\limits_{|\a|=m} c_\a (x'') \,\th ^\a\]
with coefficients $c_\a (x'')$ in $S(\bbr^{n-k-1})$ such that}
\be\label {7wer34} \intl_{\bbr}  \vp (\th, s; x'')\, s^m\, ds =P_{m} (\th, x'').\ee
We equip $S_H (\tilde Z_{n,k})$ with the induced topology of $S_e (\tilde Z_{n,k})$.
\end{definition}

The main result of this section is the following
\begin{theorem} \label {657390sw}  The restricted $k$-plane transform  $\tilde R_k$
acts as an isomorphism from $S(\rn)$  onto $S_H (\tilde Z_{n,k})$.
\end{theorem}

\subsection{Auxiliary statements}

 \begin{lemma}${}$\hfill

 {\rm (i)} If $f\in C^k (\bbr^n)$, $t\in \bbr$, then for $|\a| \le k$ and  $j\le k$,
 \be\label{aqqc}   \partial^\a_x  [f (tx/|x|)]=|x|^{-|\a|}\sum\limits_{|\gam |=1}^{|\a|} t^{|\gam|} \,h_{\a,\gam} (x/|x|)\, (\partial^\gam f)(tx/|x|),\ee
\be\label{aqqc1} \frac{\partial^j}{\partial t^j}\, [f (tx/|x|)]=\sum\limits_{|\gam|=j}  h_{\gam} (x/|x|)\, (\partial^\gam f)(tx/|x|),
\ee
\noindent where $h_{\a,\gam}$ and $ h_{\gam}$  are homogeneous polynomials independent of $f$.

{\rm (ii)}  If $g\in C^k (\bbr_+)$, $\bbr_+=(0,\infty)$, then for $1\le |\b| \le k$ and  $x\neq 0$,
\be\label{aqqc1hr}  \partial^\b_x  [g (|x|)]  =\sum\limits_{k=1}^{|\b|} |x|^{k-|\b|} \,h_{\b,k} (x/|x|)\, g^{(k)} (|x|),\ee
\noindent where $h_{\b,k}$   are homogeneous polynomials  independent of $g$.
\end {lemma}
\begin{proof}  We proceed by induction. Let $|\a|=1$, that is, $\partial^\a_x =\partial/\partial x_j$ for some $j \in \{1, 2,\ldots , n\}$. Then
\[\frac{\partial}{\partial x_j}\, [f (tx/|x|)]=t\sum\limits_{k=1}^n  (\partial_k f)(tx/|x|)\,  p_{j,k}(x),\]

\[ p_{j,k}(x)=\frac{\partial}{\partial x_j} \left [\frac{x_k}{|x|}\right ]=\frac{1}{|x|}\left \{\begin{array} {ll}
\displaystyle{-\frac{x_k x_j}{|x|^2}}  & \mbox{if $ j\neq k$,}\\
\displaystyle{1-\frac{x_k^2}{|x|^2}}  & \mbox{if $ j= k$.}\\
\end{array}
\right.\]
This gives (\ref{aqqc})  for  $|\a|=1$. Now the routine calculation shows that if (\ref{aqqc}) holds  for any  $|\a|=\ell$, then it is true for $|\a|=\ell +1$.

The proof of (\ref{aqqc1}) is easier. For $j=1$,
 \[\frac{\partial}{\partial t}\, [f (tx/|x|)]=\sum\limits_{k=1}^n  (\partial_k f)(tx/|x|)\,\frac{x_k}{|x|}.\] The general case follows by iteration.  The proof of (\ref{aqqc1hr}) is straightforward by induction.
\end{proof}

\begin{corollary}\label {iokl} Let $f\in S(\bbr^n)$, $\tilde f(\theta, t)= f(t\theta)$, where $t\in \bbr$, $\theta \in S^{n-1}$. Then for any $m\in \bbz_+$  there exist  $N\in \bbz_+$ and a constant $ c_{m,N}$ independent of $f$ such that
\bea \label{jiko} ||\tilde f||_m &\equiv& \sup\limits_{|\a|+j\le m} \sup\limits_{\theta,t}
|(1+|t|)^m |(\partial_\theta^\a \partial^j_t \tilde f)(\theta,t)|\nonumber\\
&\le& c_{m,N} ||f||_N  \equiv c_{m,N}\,\sup\limits_{|\gam|\le N} \sup\limits_{y} \
\!(1+|y|)^N |(\partial^\gam f)(y)|.\nonumber\eea
In other words,  $f \to \tilde f$ is a continuous mapping from $S(\bbr^n)$ to $S_e(Z_n)$.
\end{corollary}

\begin{corollary}\label {iokl2} The map $F_1$, which assigns to a  function $w(\theta, t) \in S_e(Z_n)$ its Fourier transform in the $t$-variable, is an automorphism of the space $S_e(Z_n)$.
\end{corollary}

\subsection{Proof of Theorem \ref{657390sw}}

We split the  proof  in several steps.

\begin{proposition} \label {657390sw1}   If $f \in S(\rn)$, then $\tilde R_k f \in S_H (\tilde Z_{n,k})$ and the map $f \to \tilde R_k f$  is continuous.
\end{proposition}
\begin{proof}  By (\ref{kkmm4539a1}) and (\ref{durtkkaz}), the function
\be \label{k4539a1az} \vp (\th, s; x'')=\intl_{\th^\perp \cap \bbr^{k+1}} f(s\th +u, x'')\, d_\th u= (Rf_{x''})(\th, s), \ee
is the usual hyperplane Radon transform in $\bbr^{k+1}$ of $f_{x''} (x')= f(x',x'')$. Hence,  (\ref{7wer34}) follows from the equalities
\[  \intl_{\bbr}  \vp (\th, s; x'')\, s^m\, ds = \intl_{\bbr}(Rf_{x''})(\th, s)\, s^m\, ds=
\intl_{\bbr^{k+1}} f(x',x'') (x' \cdot \theta)^k\, dx'.\]
  The evenness property (\ref{65679z90swu}) is a consequence of (\ref{k4539a1az}).
Furthermore, by the Projection-Slice Theorem,
\[ [ \vp (\th, \cdot; x'')]^\wedge (\eta)=[(Rf_{x''})(\th, \cdot)]^\wedge (\eta)= [f(\cdot,x'')]^\wedge (\eta \th).\]
Hence, $A: f \to \vp\!=\!\tilde R_k f$ is a composition of three mappings, specifically, $A\!=\!A_3 A_2 A_1$, where
\[ \begin{array} {llll}  A_1 : \;f(x) &\to & [f(\cdot,x'')]^\wedge (\xi')&\equiv g(\xi', x'');\\
A_2 : \; g(\xi', x'') &\to & g(\th \eta, x'')&\equiv  w(\th, \eta; x'');\\
 A_3 : \; w(\th, \eta; x'') &\to & [w(\th, \cdot; x'')]^\vee (s) &\equiv  \vp (\th, s; x'').\end{array} \]
The continuity of the operators
\[ A_1: S(\rn) \! \to \!S(\bbr^{k+1} \times \bbr^{n-k-1}), \qquad A_3:\; S_e (\tilde Z_{n,k}) \to S_e (\tilde Z_{n,k})\]
   is a consequence of the isomorphism property of the Fourier transform.
 The continuity of $A_2$ from   $S(\bbr^{k+1} \times \bbr^{n-k-1})$   to $S_e (\tilde Z_{n,k})$ follows from Corollary \ref{iokl} applied in the $\xi'$-variable.
  This gives the result.
\end{proof}

The next proposition is the most technical.
\begin{proposition} \label {mnubhrkk}  If $ \vp\in S_H (\tilde Z_{n,k})$, then the function
\be\label {pp12vdkkk}
\psi (x)\equiv \psi(x', x'')=\intl_\bbr  \vp (x'/|x'|, s; x'') \, e^{is|x'|}\, ds
\ee
belongs to $S(\rn)$ and the map $ \vp \to \psi$ is continuous.
\end{proposition}
\begin{proof} We have to show that for every $m\in \bbz_+$ there exist $M=M(m)\in \bbz_+$  and a constant $C_m>0$ independent of $\vp$ such that
$||\psi||_m \le C_m \, ||\vp||_M$. To this end, it suffices to prove the following inequalities:
\be\label {ptgtgbbkkk} \sup\limits_{|p| +|\gam| \le m}\,
\sup\limits_{|x'|<1}\, \sup\limits_{x''} \,(1\!+\!|x''|)^m \, |\partial_{x'}^p  \partial_{x''}^\gam \psi(x', x'')|\le C_m \, || \vp||_M,\ee
 \be\label {ptgtgbbkkk1} \sup\limits_{|p| +|\gam| \le m}\,
\sup\limits_{|x'|>1}\, \sup\limits_{x''} \,(1\!+\!|x'|\!+\!|x''|)^m  \,|\partial_{x'}^p  \partial_{x''}^\gam \psi(x', x'')|\le C_m \, ||\vp||_M.\ee

In the following the letters $c$ and $C$ with subscripts stand for  constants which are not necessarily the same in any two occurrences.

 STEP 1 (proof of (\ref{ptgtgbbkkk}). Fix any  $q\in \bbn$. By Taylor's formula,
 \be\label {obvjhr1kk}
e^z=\sum\limits_{\nu=0}^{q-1}\frac{z^\nu}{\nu !}+ e_q (z), \qquad e_q (z)=\sum\limits_{\nu=q}^{\infty}\frac{z^\nu}{\nu !}.
\ee
Putting $z=is|x'|$,  we have
\bea
\psi(x', x'')&=&\sum\limits_{\nu=0}^{q-1}  \frac{(i|x'|)^\nu }{\nu !}   \intl_\bbr  \vp (x'/|x'|, s; x'') \,s^\nu \, ds\nonumber\\
&+&
 \intl_\bbr  \vp (x'/|x'|, s; x'') \, e_q (is|x'|)\, ds.\nonumber\eea
By (\ref{7wer34}),
\[
(i|x'|)^\nu  \intl_\bbr  \vp (x'/|x'|, s; x'') \,s^\nu \, ds = (i|x'|)^\nu P_{\nu} (x'/|x'|, x'')\]
where $ P_{\nu} (\th, x'')\!=\!\sum_{|\a|=\nu} c_{\a} (x'') \,\th^\a$, $c_{\a} (x'')\!\in \!S(\bbr^{n-k-1})$, or, by the homogeneity,
 \be\label {ob567r1kk}  (i|x'|)^\nu  \intl_\bbr  \vp (x'/|x'|, s; x'') \,s^\nu \, ds = P_{\nu} (ix', x'').\ee
Hence, for $\gam \in \bbz^{n-k-1}_+$, we may write
 \be\label {opljhr1kk}
\partial^\gam_{x''}\psi (x', x'')=\sum\limits_{\nu=0}^{q-1}  \frac{P_{\nu,\gam} (ix', x'')}{\nu !} + \Psi_{q,\gam} (x', x''),\ee
\bea \label {opljhr1kkty1} P_{\nu,\gam} (ix', x'')&=&(i|x'|)^\nu  \intl_\bbr (\partial^\gam_{x''} \vp) (x'/|x'|, s; x'') \,s^\nu \, ds\\
\label {opljhr1kkty2}&=&\sum\limits_{|\a|=\nu} (\partial^\gam c_{\a}) (x'') (ix')^\a, \eea
   \be \label {jjyyz1kksa}\Psi_{q,\gam} (x', x'')=\intl_\bbr (\partial^\gam_{x''} \vp) (x'/|x'|, s; x'')\, e_q (is|x'|)\, ds.\ee
 Let us estimate the derivatives  $(\partial^p _{x''}\Psi_{q,\gam}) (x', x'')$, assuming  $0\le |p|<q$. For $|x'|>0$ we have
 \bea
&&\qquad\partial^p_{x'} [(\partial^\gam_{x''} \vp) (x'/|x'|, s; x'')\, e_q (is|x'|)]\nonumber\\
&&{}\nonumber\\
\label {opljhr3kk} &&\qquad=\sum\limits_{\a+\b=p} c_{\a,\b}\, \partial^\a_{x'}[
(\partial^\gam_{x''} \vp) (x'/|x'|, s; x'')] \, \partial^\b_{x'} [e_q (is|x'|)]. \qquad\qquad\eea
By (\ref{aqqc}),
 \[\partial^\a_{x'}[
(\partial^\gam_{x''} \vp) (x'/|x'|, s; x'')]\!= \!|x'|^{-|\a|}\!\sum\limits_{|\mu |=1}^{|\a|} h_{\a,\mu} (x'/|x'|)\, (\partial^\gam_{x''}\partial^\mu_{x'} \vp_0)(x'/|x'|,s; x'')\]
where $\vp_0(x',s; x'')=\vp(x'/|x'|,s; x'')$ and $ h_{\a,\mu}$  are homogeneous polynomials independent of $\vp$. Thus, by (\ref{iiuuyg}),
\bea \label {opiilhr3kk} &&|\partial^\a_{x'}[
(\partial^\gam_{x''} \vp) (x'/|x'|, s; x'')]|\\
&&\le c_p \,|x'|^{-|\a|}\sum\limits_{|\mu |=1}^{|\a|} \sup\limits_{\th} |(\partial^\gam_{x''}\partial^\mu_\th \vp)(\th, s; x'')|.\nonumber\eea
To estimate $\partial^\b_{x'} [e_q (is|x'|)]$, we consider the cases $\b\neq 0$ and $\b=0$ separately. If $\b\neq 0$, then, by
 (\ref{aqqc1hr}),
 \be\label {opljhr3bnkk}
|\partial^\b_{x'} [e_q (is|x'|)]|\le  \sum\limits_{j=1}^{|\b|} |x'|^{j-|\b|} \,|h_{\b,j} (x'/|x'|)|\, |s|^{j}\, |e_q^{(j)} (is|x'|)|\ee
 where $h_{\b,j}$  are homogeneous polynomials.
 Since $e_q^{(j)}(z)=e_{q-j}(z)$ for any $0\le j\le q$,  the function
\[
\frac{e_q^{(j)} (z)}{z^{q-j}}=\frac{e_{q-j} (z)}{z^{q-j}}=
\frac{1}{z^{q-j}}\left ( e^z- \sum\limits_{\nu=0}^{q-j-1}\frac{z^{\nu}}{\nu !}\right ), \qquad z\in \bbc,\]
is bounded (check the cases $|z|\le 1$ and $|z|>1$ separately).
Hence, the expression $(is|x'|)^{j-q}\, e_q^{(j)} (is|x'|)$  is bounded uniformly in $s$ and $x'$, and (\ref{opljhr3bnkk}) yields
 \[|\partial^\b_{x'} [e_q (is|x'|)]|\le c_{\b, q} \,\sum\limits_{j=1}^{|\b|} |x'|^{j-|\b|} \, |s|^{j}\, |sx'|^{q-j}.\]
  This gives
  \be\label {bbpljhr3bn} |\partial^\b_{x'} [e_q (is|x'|)]|\le c_{q} \,|x'|^{q-|\b|} (1\!+\!|s|)^{q}.\ee
The last estimate extends to   $\b=0$, but the proof in this case is easier. Combining (\ref {opiilhr3kk}), (\ref {bbpljhr3bn}) and
  (\ref{opljhr3kk}), and keeping in mind that $|p|<q$, we obtain
 \bea
&&|\partial^p_{x'} [\partial^\gam_{x''}\vp(x'/|x'|, s; x'')\, e_q (is|x'|)]|  \nonumber\\
&&\le c_q (1\!+\!|s|)^{q} \Big (\sum\limits_{|\mu |=1}^{q} \sup\limits_{\th} |(\partial^\mu_\th \partial^\gam_{x''}\vp)(\th,s; x'')|\Big ) \sum\limits_{\a+\b=p} |x'|^{q-|\b|-|\a|}\nonumber\\
\label {ppphr3kk} &&\le  \frac{\tilde c_q  }{(1\!+\!|s|)^2}\,   \,|x'|^{q-|p|} \, \sup\limits_{|\mu|\le q} \,\sup\limits_{\theta, s}  (1\!+\!|s|)^{q+2}  |(\partial^\mu_\th \partial^\gam_{x''}\vp)(\th,s; x'')|. \eea
Since $q, \gam$ and $|p|<q$ are arbitrary, the
 latter means that we can  differentiate under the  sign of integration in  (\ref{jjyyz1kksa}) infinitely many times. Moreover,  if we fix any  $m\in \bbz_+$ and any $q>m$,  then, by (\ref{ppphr3kk}), we obtain
  \bea
 && \!\!\!\!\sup\limits_{|p| +|\gam| \le m}\,
\sup\limits_{|x'|<1}\, \sup\limits_{x''} \,(1\!+\!|x''|)^m \, |\partial_{x'}^p  \Psi_{q,\gam}) (x', x'')|\nonumber\\
&&\!\!\!\!\le \sup\limits_{|p| +|\gam| \le m}\,
\sup\limits_{|x'|<1}\, \sup\limits_{x''} \,(1\!+\!|x''|)^m \!\! \intl_{\bbr} \! |\partial^p_{x'} [\partial^\gam_{x''}\vp(x'/|x'|, s; x'')\, e_q (is|x'|)]|\, ds \qquad \nonumber\\
&&\!\!\!\! \le   C_q \, \sup\limits_{|p| +|\gam| \le m}\,\sup\limits_{|x'|<1}\,|x'|^{q-|p|} \,\sup\limits_{x''} \,(1\!+\!|x''|)^m \nonumber\\
&& \!\!\!\!\times \sup\limits_{|\mu|\le q} \,\sup\limits_{\theta, s} \, (1\!+\!|s|)^{q+2}  |(\partial^\mu_\th \partial^\gam_{x''}\vp)(\th,s; x'')|\nonumber\\
&& \!\!\!\!\le  C_q \,\sup\limits_{|p| +|\gam| \le m}\,\sup\limits_{|x'|<1}\,|x'|^{q-|p|} \, \sup\limits_{|\mu|+|\gam|\le m+q+2} \nonumber\\
&& \!\!\!\!\times \sup\limits_{\theta, s, x''} \,(1\!+\!|s|\!+\!|x''|)^{m +q+2} |(\partial^\mu_\th \partial^\gam_{x''}\vp)(\th,s; x'')|\nonumber\\
&& \!\!\!\!\le  C_q \,\sup\limits_{|p| +|\gam| \le m}\,\sup\limits_{|x'|<1}\,|x'|^{q-|p|} \, ||\vp||_{m+q+2}.\nonumber\eea
Setting $q=m+1$, we get
\be\label {bb4444bn} \sup\limits_{|p| +|\gam| \le m}\,
\sup\limits_{|x'|<1}\, \sup\limits_{x''} \,(1\!+\!|x''|)^m \, |\partial_{x'}^p  \Psi_{m+1,\gam}) (x', x'')| \le  C_m  \, ||\vp||_{2m+3}.\ee

Let us estimate the derivatives of the first term in (\ref{opljhr1kk}). By (\ref{opljhr1kkty1}) and  (\ref{opljhr1kkty2}),
\bea
(\partial^p_{x'} P_{\nu,\gam}) (ix', x'')&=&  \partial^p_{x'} \Big[ (i|x'|)^\nu  \intl_\bbr (\partial^\gam_{x''}\vp) (x'/|x'|, s; x'') \,s^\nu \, ds\Big ]\nonumber\\
&=&\sum\limits_{|\a|=\nu} (\partial^\gam c_{\a}) (x'')\, \partial^p_{x'}[(ix')^\a], \quad \nu \le q-1. \nonumber\eea
Hence, if $|p|> \nu$, then  $(\partial^p_{x'} P_{\nu,\gam}) (ix', x'')=0$. Suppose $|p|\le \nu$. Then
\bea
 && \sup\limits_{|p| +|\gam| \le m}\,
\sup\limits_{|x'|<1}\, \sup\limits_{x''} \,(1\!+\!|x''|)^m \, |(\partial_{x'}^p   P_{\nu,\gam}) (ix', x'')|\nonumber\\
&&\le c_p\, \sup\limits_{|p| +|\gam| \le m}\,
\sup\limits_{|x'|<1}\, \sup\limits_{x''} \,(1\!+\!|x''|)^m \nonumber\\
&& \times \sum\limits_{\a+\b=p} |\partial_{x'}^\b  [(i|x'|)^\nu]| \,
\Big | \intl_\bbr \partial_{x'}^\a [\partial^\gam_{x''} \vp) (x'/|x'|, s; x'')] \,s^\nu \, ds\Big |. \nonumber\eea
By (\ref{aqqc1hr}) and (\ref{opiilhr3kk}), the last expression does not exceed the following:
\bea
&&c_{p,\nu}\, \sup\limits_{|p| +|\gam| \le m}\,
\sup\limits_{|x'|<1}\, \sum\limits_{\a+\b=p} |x'|^{\nu-|\b|-|\a|}\, \sup\limits_{x''} \,(1\!+\!|x''|)^m \nonumber\\
&&\times \intl_\bbr
\Big [\sum\limits_{|\mu |=1}^{|\a|} \sup\limits_{\th} |(\partial^\gam_{x''}\partial^\mu_\th \vp)(\th, s; x'')|\Big ]\,s^\nu \, ds \nonumber\\
&&\le
\tilde c_{p,\nu}\, \sup\limits_{|p| +|\gam| \le m}\,\sup\limits_{|x'|<1}\,|x'|^{\nu-|p|}\,\sup\limits_{s,x''} \,
(1\!+\!|x''|\!+\!|s|)^{m +\nu +2}\nonumber\\
&&\times
\sum\limits_{|\mu |=1}^{|\a|} \sup\limits_{\th} |(\partial^\gam_{x''}\partial^\mu_\th \vp)(\th, s; x'')| \nonumber\\
&&\le C_{p,\nu}\, \sup\limits_{|p| +|\gam| \le m}\,\sup\limits_{|x'|<1}\,|x'|^{\nu-|p|}\,\sup\limits_{|\mu|+|\gam|+j\le |p|+m+\nu}\nonumber\\
&&\times \sup\limits_{\th,s,x''} \,(1\!+\!|x''|\!+\!|s|)^{m +\nu +2+|p|}|\,(\partial^\gam_{x''}\partial^\mu_\th \partial^j_s\vp)(\th, s; x'')| \nonumber\\
&&\le C_{p,\nu}\,  \sup\limits_{|p| +|\gam| \le m}\,\sup\limits_{|x'|<1}\,|x'|^{\nu-|p|}\,||\vp||_{m +\nu +2+|p|}\nonumber\\
&&\le C_{p,\nu}\, ||\vp||_{2m +\nu +2}\le  C_{m,q}\, ||\vp||_{2m +q +1}=  C_m\, ||\vp||_{3m +2};\nonumber\eea
here we choose $q=m+1$, as in (\ref{bb4444bn}).

Combining the last estimate with (\ref{opljhr1kk}) and (\ref{bb4444bn}), we obtain
\[
\sup\limits_{|p| +|\gam| \le m}\,
\sup\limits_{|x'|<1}\, \sup\limits_{x''} \,(1\!+\!|x''|)^m \,|\partial_{x'}^p\partial^\gam_{x''}\psi (x', x'')|\le C_m \, (||\vp||_{2m+3}+||\vp||_{3m +2}).\] This gives the first required inequality   (\ref{ptgtgbbkkk}).

\vskip 0.3 truecm

STEP 2.   Let us prove (\ref{ptgtgbbkkk1}).  Fix any $m\in \bbz_+$.
Then integration by parts yields
\[ [\vp (\th, \cdot; x'')]^\wedge (\eta)=(-i\eta)^{-m}\intl_\bbr (\partial^m_s  \vp )(\th, s; x'') \,e^{is\eta}\, ds.\]
For arbitrary multi-indices $\gam \in \bbz^{n-k-1}_+$ and $p\in \bbz^{k+1}_+$  satisfying $|p| +|\gam| \le m$, we have
\[  (\partial_{x'}^p\partial^\gam_{x''}\psi) (x', x'')=\partial_{x'}^p [(-i|x'|)^{-m}\psi_{m, \gam} (x', x'')], \]
\[ \psi_{m, \gam} (x', x'')=\intl_\bbr (\partial^m_s \partial^\gam_{x''} \vp)(x'/|x'|,s; x'')\, e^{is|x'|}\, ds,\]
and therefore,
\be\label {opl000vili}
|(\partial^p_{x'}  \partial^\gam_{x''}\psi)(x', x'')|\le c_p\,\sum\limits_{u+v=p} |\partial_{x'}^u [|x'|^{-m}]|\, |(\partial_{x'}^v \psi_{m, \gam})(x', x'')|.
\ee
By (\ref{aqqc1hr}),
 \be \label {owercvr771}
  |\partial_{x'}^u [|x'|^{-m}]|\le c_{m,u} \, |x'|^{-m-|u|}.\ee
To estimate $|(\partial_{x'}^v \psi_{m, \gam})(x', x'')|$, as in STEP 1, we have
\bea
&&\partial^v_{x'} [(\partial^m_s  \partial^\gam_{x''} \vp)(x'/|x'|, s; x'')\, e^{is|x'|}]\nonumber\\
\label {opljhr3vili} &&=\sum\limits_{\a+\b=v} c_{\a,\b}\, \partial^\a_{x'} [(\partial^m_s \partial^\gam_{x''} \vp)(x'/|x'|, s; x'')] \, \partial^\b_{x'} [e^{is|x'|}],\qquad \eea
where
\bea \label {opiilhr3vrvr} &&|\partial^\a_{x'} [(\partial^m_s \partial^\gam_{x''} \vp)(x'/|x'|, s; x'')]|\\
&&\le c_v \,|x'|^{-|\a|}\sum\limits_{|\mu |=1}^{|\a|} \sup\limits_{\th} |(\partial^\mu_\th \partial^m_s \partial^\gam_{x''}\tilde \vp)(\th,s; x'')|;\nonumber\eea
cf.  (\ref{opiilhr3kk}). Furthermore, by (\ref{aqqc1hr}), for $\b\neq 0$ we have
\[ \partial^\b_{x'} [e^{is|x'|}]=\sum\limits_{j=1}^{|\b|} |x'|^{j-|\b|} \,h_{\b,j} (x'/|x'|)\, (is)^j\, e^{is|x'|},\]
where $h_{\b,j}$   are homogeneous polynomials. Hence,  since $|x'|>1$,
\be \label {owercvr}   |\partial^\b_{x'} [e^{is|x'|}]|\le c_\b \, \sum\limits_{j=1}^{|\b|} |x'|^{j-|\b|}  |s|^j\le \tilde c_\b \, (1\!+\!|s|)^{|\b|}.\ee
This estimate obviously holds if $\b=0$. Combining (\ref{opljhr3vili}), (\ref{opiilhr3vrvr}), and (\ref{owercvr}), for $|v|\le |p|$ we obtain
\bea
&&|\partial^v_{x'} [(\partial^m_s \partial^\gam_{x''}\vp)(x'/|x'|, s; x'')\, e^{is|x'|}]|\nonumber\\
&&\le c_v \, (1\!+\!|s|)^{|v|}\sum\limits_{|\mu |=1}^{|v|} \sup\limits_{\th} |(\partial^\mu_\th \partial^m_s \partial^\gam_{x''}\vp)(\th,s; x'')|\nonumber\\
 &&\le\frac{c_p  }{(1\!+\!|s|)^2}\,\sup\limits_{|\mu|+j+|\gam|\le 2m+2} \; \sup\limits_{\th,s}   (1\!+\!|s|\!+\!|x''|)^{2m+2}|(\partial^\mu_\th \partial^j_s \partial^\gam_{x''} \vp)(\th,s; x'')|.\nonumber\eea
 Hence,  we can  differentiate under the  sign of integration in  $\psi_{m, \gam}(x', x'')$
 and get
    \bea \label {owercvr77}
 && |(\partial_{x'}^v \psi_{m, \gam})(x', x'')|\\
 &&\le \tilde c_p \, \sup\limits_{|\mu|+j+|\gam|\le 2m+2}    \sup\limits_{\th,s}   (1\!+\!|s|\!+\!|x''|)^{2m+2}|(\partial^\mu_\th \partial^j_s \partial^\gam_{x''} \vp)(\th,s; x'')|.\nonumber\eea
  Thus, (\ref{opl000vili}), (\ref{owercvr771}), and  (\ref{owercvr77}) yield
\bea \label {owercvr778} && \sup\limits_{|p| +|\gam| \le m}\,
\sup\limits_{|x'|>1}\, \sup\limits_{x''} \,(1\!+\!|x'|\!+\!|x''|)^m  \,|\partial_{x'}^p  \partial_{x''}^\gam \psi(x', x'')|\\
&&\le   c_{m} \,\sup\limits_{|p| +|\gam| \le m}\,
\sup\limits_{|x'|>1}\,\sup\limits_{x''} \,(1\!+\!|x'|\!+\!|x''|)^m\, \sum\limits_{u+v=p}  |x'|^{-m-|u|}\nonumber\\
&& \times \sup\limits_{|\mu|+j+|\gam|\le 2m+2} \; \sup\limits_{\th,s} \,  (1\!+\!|s|\!+\!|x''|)^{2m+2}|(\partial^\mu_\th \partial^j_s \partial^\gam_{x''}\vp)(\th,s; x'')|.\nonumber\eea
 Since for $|x'|>1$,
 \[\frac{1+|x'|+|x''|}{|x'|}=1+\frac{1+|x''|}{|x'|}<2+|x''|,\]
 then the expression in (\ref{owercvr778}) does not exceed
 \[ C_m\,\sup\limits_{|\mu|+j+|\gam|\le 3m+2}\, \sup\limits_{\th,s, x''} (1\!+\!|s|\!+\!|x''|)^{3m+2}\,|(\partial^\mu_\th \partial^j_s \partial^\gam_{x''} \vp)(\th,s; x'')|,\]
which is $ C_m\, || \vp ||_{3m+2}$. This completes the proof of Proposition \ref{mnubhrkk}.
\end{proof}

\subsubsection{The end of the proof of Theorem \ref{657390sw}}

In view of Proposition \ref{657390sw1}, it remains  to prove that any $\vp \in S_H (\tilde Z_{n,k})$ is uniquely represented as $\vp=\tilde R_kf$ for some $f\in S(\rn)$ and the map $\vp \to f$ is continuous in the topology of the corresponding Schwartz spaces.
In the following, dealing with a function of $x=(x', x'')$ on $\rn=\bbr^{k+1}\times \bbr^{n-k-1}$, we denote by $F_1 [\cdot](y')$ and $F_2 [\cdot](y'')$ the Fourier transform of this function in the first and the second variable, respectively. The corresponding $n$-dimensional Fourier transform will be denoted by  $F_n [\cdot](y)$, $y=(y', y'')$.

If $\vp \in S_H (\tilde Z_{n,k})$, then the function
$\psi(y',x'')= [ \vp (y'/|y'|, \cdot ; x'')]^\wedge (|y'|)$ belongs to $S(\bbr^{k+1}\times \bbr^{n-k-1})=S(\bbr^n)$; see Proposition \ref{mnubhrkk}. We define a new function
\[ \psi_1(y)\equiv \psi_1(y',y'')= F_2 [\psi(y', \cdot)](y'').\]
Putting $y'=\eta \th$,   $\eta \in \bbr$, $\th \in S^k$, according to (\ref{pp12vdkkk})  and (\ref{65679z90swu}),  we have
\be \label {jj4356z2a3}  F_2^{-1} [ \psi_1(\eta \th, \cdot) ](x'')= \psi (\eta \th,x'') = [ \vp (\th, \cdot ; x'')]^\wedge (\eta).\ee

 Now, let $f\in S(\rn)$ be the inverse $n$-dimensional Fourier transform of $\psi_1$, i.e.,$f=F_n^{-1}\psi_1=  F_1^{-1}F_2^{-1} \psi_1$. Then  by (\ref{kkmm4539a1bu}) and the Projection-Slice Theorem,
\bea
[(\tilde R_k f)(\th, \cdot ; x'')]^\wedge (\eta)&=&[(R f_{x''})(\th, \cdot)]^\wedge (\eta)=[f(\cdot ; x'')]^\wedge (\eta\th)\nonumber\\
&=&[(F_1^{-1}F_2^{-1} \psi_1)(\cdot ; x'')]^\wedge (\eta\th)\nonumber\\
&=& [F_2^{-1} [\psi_1 (\eta\th, \cdot)] (x'') \qquad \mbox{\rm (use (\ref{jj4356z2a3}))}\nonumber\\
&=&  [\vp (\th, \cdot ; x'')]^\wedge (\eta).\nonumber\eea
It follows that  $\tilde R_k f=\vp$. Moreover,  the map $\vp \to f$ is continuous thanks to the continuity of the
  mappings
 $$  \vp  \longrightarrow   \psi \longrightarrow   \psi_1 \longrightarrow  F_n^{-1} \psi_1=f$$
 in the corresponding Schwartz spaces.


\end{document}